\documentclass{amsart}

\usepackage{graphicx} 
\usepackage{amsmath,amssymb,amsthm,mathtools,thmtools} 
\usepackage{xspace}
\usepackage{enumitem}
\usepackage{bbm}
\usepackage{xcolor}
\usepackage{sepfootnotes}
\usepackage{verbatim}

\usepackage[pdftex,pagebackref=false]{hyperref}
\usepackage[capitalise]{cleveref}

\newcommand{\inv}[1]{{#1^{-1}}}
\newcommand{\dee}{\,\mathrm{d}}
\newcommand{\indic}[1]{{\mathbbm{1}}_{#1}}
\newcommand{\haar}{\mathbf{m}}

\newcommand{\dirac}[1]{\delta_{#1}}
\newcommand{\invo}[1]{{#1}^*}
\newcommand{\supp}{\operatorname{supp}}

\newcommand{\C}{{\mathbb{C}}}

\newcommand{\Z}{{\mathbb{Z}}}

\newcommand{\OP}{{\mathbb{Z}_p}} 
\newcommand{\QP}{{\mathbb{Q}_p}}

\newcommand{\falg}[1]{\operatorname{A}(#1)}

\newcommand{\Lp}[2]{\operatorname{L}^{#1}\ifnonempty{#2}{(#2)}{}}
\newcommand{\Lo}[1]{\Lp{1}{#1}} 
\newcommand{\Li}[1]{\Lp{\infty}{#1}} 
\newcommand{\lp}[2]{\operatorname{\ell}^{#1}\ifnonempty{#2}{(#2)}{}}
\newcommand{\lo}[1]{\lp{1}{#1}} 
\newcommand{\meas}[1]{\operatorname{M}(#1)}
\newcommand{\prob}[1]{\operatorname{Prob}(#1)}

\newcommand{\cbdd}[1]{\operatorname{C}_b({#1})} 
\newcommand{\cc}[1]{\operatorname{C}_c({#1})} 
\newcommand{\cvan}[1]{\operatorname{C}_0({#1})} 

\newcommand{\sdp}[2]{{#1 \rtimes #2}}
\newcommand{\actson}{\curvearrowright}
\newcommand{\orbits}[2]{\mathcal{O}_{#2}(#1)}

\newcommand{\stab}[2]{{#2}_{#1}}
\newcommand{\lac}[2]{{#2} \mathrel{{\boldsymbol{\cdot}}} {#1}}

\newcommand{\cLo}[1]{\operatorname{ZL}^1(#1)}
\newcommand{\cfalg}[1]{\operatorname{ZA}(#1)}
\newcommand{\kfalg}[2]{\mathop{\mathrm{Z}_{#2}\mathrm{A}}(#1)}
\newcommand{\kLo}  [2]{\mathop{\mathrm{Z}_{#2}\mathrm{L}^1}(#1)}
\newcommand{\klo}  [2]{\mathop{\mathrm{Z}_{#2}\mathrm{\ell}^1}(#1)}

\newcommand{\kmeas}[2]{\mathop{\mathrm{Z}_{#2}\mathrm{M}}(#1)}
\newcommand{\hdual}[1]{\widehat{#1}}

\newcommand{\pdual}[1]{\widehat{#1}}
\newcommand{\ft}[1]{\widehat{#1}}
\newcommand{\conj}[1]{\overline{#1}}
\newcommand{\ann}[1]{{#1}^\perp}

\newcommand{\Ri}{R}
\newcommand{\Riu}{\Ri^*}
\newcommand{\Mi}{M}

\newcommand{\Fi}{F}
\newcommand{\GL}[2]{\operatorname{GL}_{#2}(#1)}

\newcommand{\gf}[1]{\mathbb{F}_{#1}}

\newcommand{\Inn}[1]{\operatorname{Inn}(#1)}

\newcommand{\F}{\mathbb{F}}

\newcommand{\irrep}[1]{\widehat{#1}}
\newcommand{\conjhyp}[1]{\operatorname{Conj}(#1)}

\newcommand{\extrem}[1]{\operatorname{E}(#1)}
\newcommand{\pdefu}[1]{\operatorname{P}_1(#1)}

\DeclarePairedDelimiterX{\set}[1]{\{}{\}}{\argordot{#1}}
\DeclarePairedDelimiterX{\norm}[1]{\|}{\|}{\argordot{#1}}
\DeclarePairedDelimiterX{\abs}[1]{\lvert}{\rvert}{\argordot{#1}}
\DeclarePairedDelimiterX{\card}[1]{\lvert}{\rvert}{\argordot{#1}}
\DeclarePairedDelimiterX{\gen}[1]{\langle}{\rangle}{\argordot{#1}}

\DeclarePairedDelimiterX{\ip}[2]{\langle}{\rangle}{\argordot{#1}\,|\,\argordot{#2}} 
\DeclarePairedDelimiterX{\dup}[2]{\langle}{\rangle}{\argordot{#1}, \argordot{#2}} 
\DeclarePairedDelimiterX{\sbf}[2]{[}{]}{\argordot{#1}, \argordot{#2}} 

\newtheorem{theorem}{Theorem}[section]
\newtheorem{lemma}[theorem]{Lemma}
\newtheorem{proposition}[theorem]{Proposition}
\newtheorem{conjecture}[theorem]{Conjecture}
\newtheorem{corollary}[theorem]{Corollary}

\theoremstyle{definition} 
\newtheorem{definition}[theorem]{Definition}

\theoremstyle{remark}

\newtheorem{remark}[theorem]{Remark}

\newcommand{\define}[1]{{\itshape #1}}
\newcommand{\deq}{:=}

\newcommand{\AV}{Aleksa Vuji\v{c}i\'{c}\xspace}

\newcommand\restr[2]{{
  \left.\kern-\nulldelimiterspace 
  #1 
  \vphantom{\big|} 
  \right|_{#2} 
  }}

\newcommand*{\ifnonempty}[3]{%
    \def\arg{#1}%
    \ifx\arg\empty
        #3%
    \else
        #2%
    \fi%
}

\newcommand*{\argordefault}[2]{\ifnonempty{#1}{#1}{#2}}
\newcommand*{\argordot}[1]{\argordefault{#1}{\blankdot}}
\newcommand{\blankdot}{\,\cdot\,}

\title{The Fourier algebra of certain compact orbit hypergroups}
\author{\AV}
\date{\today}

\begin{document}

\begin{abstract}

    Let $R$ be a compact discrete valuation ring and $G = \sdp{R^d}{\GL{R}{d}}$.
    We show that the central Fourier algebra $\cfalg{G}$ is not amenable, reaffirming a conjecture of Alaghmandan and Spronk.
    Our methods reduce to studying the Fourier algebra of the commutative orbit hypergroup $H = R^d/\GL{R}{d}$. 
    Along the way, we also show that dual of any commutative orbit hypergroup $H$ is the hypergroup of the corresponding dual action, and this in turn shows that $\falg{H} \cong \Lo{\hdual{H}}$, which aligns with the classical setting.
\end{abstract}

\maketitle

\section{Introduction}


Given a compact group $G$, we shall let $\irrep{G}$ denote the (equivalence classes of) irreducible unitary representations.
Each representation has a well-defined character associated to it, and the closed span of these characters coincides with $\cLo{G}$, the centre of the $\Lo{G}$.
The amenability and weak amenability of $\cLo{G}$ was initially studied by Azimifard, Samei and Spronk \cite{azimifard-samei-spronk}, and it was later shown by Alaghmandan and Crann \cite{alaghmandan-crann} that $\cLo{G}$ is amenable whenever $G$ is virtually abelian (has an abelian subgroup of finite index).
Partial converses were obtained for connected groups, and groups which are an (infinite) product of finite non-abelian groups \cite[Theorem 1.7, Theorem 1.10]{azimifard-samei-spronk}.
It is conjectured that the converse holds for all compact groups.

The appropriate dual notion for $\cLo{G}$ is the \emph{central Fourier algebra} $\cfalg{G}$.
Despite the name and suggestive notation, this is clearly not the centre of the Fourier algebra, rather this is a holdover from its connection to $\cLo{G}$.
Since the algebra $\cLo{G}$ consists precisely of the $\Lo{}$-functions which are (almost everywhere) constant on conjugacy classes, this is naturally isomorphic to $\Lo{H}$, where $H$ is the \emph{conjugacy class hypergroup} $\conjhyp{G}$.
The dual hypergroup of $\conjhyp{G}$ is the character hypergroup $\irrep{G}$, as shown in \cite[Proposition 2.4]{alaghmandan-amini}, and through the usual Fourier transform, it is seen that $\Lo{\irrep{G}}$ is Banach algebra isomorphic to the central Fourier algebra $\cfalg{G} = \cLo{G} \cap \falg{G}$, which is considered as a subalgebra of $\falg{G}$.
This algebra was initially introduced by Alaghmandan and Spronk \cite{alaghmandan-spronk}, where they show that $\cfalg{G}$ is amenable whenever $G$ is virtually abelian (and conjecture the converse to be true as well).

The $K$-central Fourier algebra is a generalisation of the central Fourier algebra, and it can also be framed as the Fourier algebra of its associated \emph{orbit hypergroup}.
We introduce this terminology and examine this connection in \cref{sec:hypers}.
These $K$-central Fourier algebras arise particularly when studying semidirect product groups;
they are typically simpler and therefore often easier to compute directly.
We show this explicitly in \cref{sec:affine}, where we prove that the for instance, $p$-adic affine group has a non-amenable central Fourier algebra, reinforcing the conjecture of Alaghmandan and Spronk.

Throughout this paper, unless stated otherwise, $G$ will denote a compact group, and $\Gamma$ a discrete group (typically serving as a dual to $G$).
We also reserve $K$ for a compact group acting on either $G$ or $\Gamma$.

\section{Orbit Hypergroups}
\label{sec:hypers}

\subsection{Hypergroups}
\emph{Hypergroups} are a generalisation of groups which were originally introduced by Jewett \cite{jewett} under the name ``\emph{convolution spaces}'' or ``\emph{convos}''.
Let $H$ denote a locally compact Hausdorff space, and let us adopt the following conventions.
\begin{itemize}
    \item Let $\meas{H}$ be the complex-valued (Radon) measures on $H$.
    
    \item Let $\prob{H} \subseteq \meas{H}$ be the space of probability measures.
    
    \item For $x \in H$, let $\dirac{x}$ denote the point measure at $x$.
    
    
\end{itemize}
        
        
        
        
        
While we omit the definition of a general hypergroup, it may be found in the original work of Jewett or of the text of Bloom and Heyer \cite{bloom-heyer,jewett}.
Instead we present a few key ideas below, as well as a definition for a specific class of hypergroups.
Additional details may also be found in the aforementioned texts.

Given an $x \in H$, we shall often let $x$ stand in place of the point measure $\dirac{x}$. 
In particular we will write $x * y = \dirac{x} * \dirac{y}$, and in this case, we shall call the map $* : H \times H \to \prob{H}$ the \define{hyperproduct} of $H$.
For subsets $A,B \subseteq H$, we define $A * B$ as a subset of $H$ given by
\begin{equation}
    \label{eq:set-hypproduct}
    A * B \deq \bigcup_{x \in A, y \in B} \supp(x*y).
\end{equation}

In this paper, we shall solely use hypergroups which possess a Haar measure, which shall be denoted $\haar_H$.
Naturally, we set $\Lo{H} \deq \Lo{H,\haar_H}$ and equip it with the Banach algebra structure inherited from $\meas{H}$.

When $H$ is commutative, its \define{dual} $\hdual{H}$ is defined as the collection of all non-zero bounded continuous functions $\chi : H \to \C$ satisfying $\chi(x * y) = \chi(x)\chi(y)$ and $\chi(\invo{x}) = \conj{\chi(x)}$.
This space is equipped with the topology of uniform convergence on compacta, and has an associated Plancherel measure (see \cite[Theorem 7.3I]{jewett}).
However, in general, this space need not be a hypergroup; when $\hdual{H}$ does form a hypergroup where the Haar measure and Plancherel measure coincide, we will say that $H$ is a \define{strong} hypergroup.
All hypergroups presented in this paper will be strong hypergroups.

The Fourier algebra of a hypergroup is also somewhat nuanced.
It is defined in much the same manner as it is for classical groups, so we shall not present the definition here; instead we refer the reader to any of the papers \cite{alaghmandan,alaghmandan-crann,muruganandam1}.
However, $\falg{H}$ is not necessarily closed under pointwise multiplication, let alone a Banach algebra.
For hypergroups in which $\falg{H}$ does indeed form a Banach algebra, we say that $H$ is a \define{regular Fourier hypergroup}.

\subsection{Orbit hypergroups}
\newcommand{\SH}[1]{\textbf{SH}\textsubscript{#1}}
While it has long been known that the orbits of a compact group $K$ of automorphisms on $G$ form a hypergroup (see the original paper of Jewett \cite[Section 8.3]{jewett}), it will simplify matters to describe them in the context of \define{spherical hypergroups}.
These were introduced and studied by Muruganandam, and they are defined in terms of a mapping $\pi : \cc{G} \to \cc{G}$ known as a \define{spherical projector}, which satisfies conditions labelled as \SH{1}, \SH{2}, and \SH{3} in \cite[Definition 2.1]{muruganandam2}.
These are as follows:
\begin{definition}
    Let $G$ be a locally compact group.
    We say that $\pi: \cc{G} \to \cc{G}$ is a \define{spherical projector} if $\pi$ satisfies all of the following conditions.
    For convenience, we set $I(f) = \int_G f(x) \dee x$ for $f \in \cc{G}$.
    \begin{itemize}
        \item[(\SH{1})] 
            For every $f,g \in \cc{G}$,
            \begin{enumerate}[label=(\roman*)]
                \item $\pi^2 = \pi$ and preserves positivity,
                \item $\pi(\pi(f) \cdot g) = \pi(f) \cdot \pi(g)$,
                \item $I(\pi(f) \cdot g) = I(f \cdot \pi(g))$, and
                \item $I(\pi(f)) = I(f)$.
            \end{enumerate}
        \item[(\SH{2})] 
            For every $f,g \in \cc{G}$, $\pi(\pi(f) * \pi(g)) = \pi(f) * \pi(g)$.
        \item[(\SH{3})] 
            Let $\pi$ also denote the adjoint map on $\meas{G}$, and for $x \in G$, denote $O_x = \supp(\pi(\dirac{x}))$.
            For every $x,y \in G$,
            \begin{enumerate}[label=(\roman*)]
                \item either $O_x \cap O_y = \emptyset$ or $O_x = O_y$,
                \item if $x \in O_y$ then $\inv{x} \in O_{\inv{y}}$,
                \item if $O_{xy} = O_e$ then $O_x = O_{\inv{y}}$, and
                \item the mapping $x \mapsto O_x$ under the so-called ``Michael topology'' (defined in \cite[Section 2.5]{jewett}) is continuous.
            \end{enumerate}
    \end{itemize}
\end{definition}
We are primarily interested in the case when $K$ is a compact group acting on $G$, whereupon
\begin{equation}
    \label{eq:action-spherical-proj}
    \pi(f)(x) \deq \int_K f(\lac{x}{k}) \dee k
\end{equation}
for $f \in \cc{G}$ defines a spherical projector on $G$.
\begin{proposition}
    As defined above, $\pi$ is a spherical projector.
\end{proposition}
\begin{proof}
    Verifying that \eqref{eq:action-spherical-proj} satisfies \SH{1} and \SH{2} is a series of straightforward algebraic computations.
    For \SH{3}, it is clear that $\supp \pi(\delta_x)$ is precisely the $K$-orbit of $x$, from which it immediately follows that $\pi$ satisfies (i), (ii) and (iii) of \SH{3}.
    Lastly, one may check that the Michael topology, when restricted to the orbits of $K$, coincides with the usual quotient topology.
    Thus the map $x \mapsto O_x$ is continuous.
\end{proof}
We note that the spherical projector may be extended in the obvious way to maps $\pi_0 : \cvan{G} \to \cvan{G}$, $\pi^1 : \Lo{G} \to \Lo{G}$, and $\pi^A : \falg{G} \to \falg{G}$, with corresponding adjoint maps $\pi_0^* :\meas{G} \to \meas{G}$ and $\pi^\infty : \Li{G} \to \Li{G}$.
It is easy to verify that these coincide where appropriate (taking the form as in \eqref{eq:action-spherical-proj}), so we shall denote all these maps as $\pi$, which shall cause no confusion.

Muruganandam \cite[Theorem 2.12]{muruganandam2} shows that the collection $\set{O_x : x \in G}$ (which coincides with the orbit space $\orbits{G}{K}$) forms a hypergroup $H$, with convolution inherited from $G$.
We shall call $H$ the \define{orbit hypergroup} of $K \actson G$.
For a given measure $\mu \in \meas{G}$, we say that $\mu$ is \define{$K$-radial} if $\pi(\mu) = \mu$, and we let $\kmeas{G}{K}$ denote the collection $K$-radial measures.
We can then identify the measure space $\meas{H}$ as a quotient of $\meas{G}$, under the restriction map onto $K$-invariant subsets of $G$.
In other words, we have that $\meas{H} = \pi(\meas{G}) = \kmeas{G}{K}$.

For $f$ a function in either of $\Lo{G}$ or $\falg{G}$, we also say that $f$ is \define{$K$-radial} if $\pi(f) = f$.
Note that this condition is equivalent to statement that $f$ is (almost everywhere) constant on the orbits of $K$. 
The collection of $K$-radial functions is similarly denoted $\kLo{G}{K}$ and $\kfalg{G}{K}$ respectively, and we call these the \define{$K$-centre of $\Lo{G}$} and the \define{$K$-central Fourier algebra of $G$}.

\begin{remark}
    It should be noted that the ``$K$-radial'' terminology slightly differs from the paper of Muruganandam.
    They use the term ``$\pi$-radial'' instead of $K$-radial, and opts for notation such as $\mathrm{A}_\pi(G)$.
\end{remark}

These $K$-central Fourier algebras also exhibit a property akin to Herz's restriction theorem.
\begin{proposition}
    \label{thm:hyper-herz}
    Let $K$ act on a locally compact group $G$, and let $N$ be a closed $K$-invariant subgroup of $G$.
    Then
    \begin{equation*}
        \restr{\kfalg{G}{K}}{N} = \kfalg{N}{K}
    \end{equation*}
\end{proposition}
\begin{proof}
    Let $\pi_G : \falg{G} \to \falg{G}$ and $\pi_N : \falg{N} \to \falg{N}$ denote the $K$-spherical projectors for $G$ and $N$ respectively, and let $R : \falg{G} \to \falg{N}$ denote the restriction map.

    The spherical projectors will be given by a formula as in \eqref{eq:action-spherical-proj}, though they will be defined on their respective Fourier algebras.
    It follows from this that $R \pi_G = \pi_N R$.
    Thus if $u \in \kfalg{G}{K}$ and $v = Ru$, then 
    \begin{equation*}
        \pi_N v = \pi_N R u = R \pi_G u = R u = v
    \end{equation*}
    and so $Ru \in \kfalg{N}{K}$.
    Surjectivity follows by Herz's restriction theorem.
\end{proof}

\subsection{Discrete, compact, and abelian orbit hypergroups}
When the underlying group $\Gamma$ is discrete, we have a concrete description of the hypergroup $H$.
\begin{proposition}
    \label{thm:discrete-hypergroup-dirac}
    Let $K$ act on a discrete group $\Gamma$, with associated spherical projector $\pi$.
    Then $\pi(\dirac{x}) \in \klo{\Gamma}{K}$ with
    \begin{equation*}
        \pi(\dirac{x}) = \frac{1}{\card{\lac{x}{K}}} \indic{\lac{x}{K}} 
    \end{equation*}
\end{proposition}
\begin{proof}
    Since $\lac{x}{K}$ is finite, it follows by the orbit stabiliser theorem that $K/\stab{x}{K}$ is finite and hence
    \begin{align*}
        \pi(f)(y)
        = \int_K f(\lac{y}{k}) \dee k
        = \frac{1}{\card{K/\stab{y}{K}}} \sum_{l \stab{y}{K} \in K/\stab{y}{K}} f(\lac{y}{l})
        = \frac{1}{\card{\lac{y}{K}}} \sum_{z \in \lac{y}{K}} f(z)
    \end{align*}
    for $f \in \lo{\Gamma}$.
    Thus we have
    \begin{equation*}
        \pi(\dirac{x})(y)
        = \frac{1}{\card{\lac{y}{K}}} \sum_{z \in \lac{y}{K}} \dirac{x}(z)
        = \frac{1}{\card{\lac{x}{K}}} \indic{\lac{x}{K}}(y) \qedhere
    \end{equation*}
\end{proof}
In particular, for a discrete group $\Gamma$, the Haar measure $\haar_H$ on $H$ is given by
\begin{equation}
    \haar_H(\pi(\dirac{x})) = \card{\lac{x}{K}}
\end{equation}
which follows as an application of \cite[Proposition 2.14]{muruganandam2}.

Suppose now that $G$ is compact.
There is a special case in which $K$ is the inner automorphisms of $G$.
In this case, we have that the algebra $\kLo{G}{K}$ aligns precisely with the algebraic centre of $\Lo{G}$.
In a similar manner, $\kfalg{G}{K}$ coincides with \emph{central Fourier algebra} $\cfalg{G}$ as originally defined and studied by Alaghmandan and Spronk \cite{alaghmandan-spronk}.
Note though this algebra has no connection with the centre of $\falg{G}$, which is of course trivial.

In particular, we are interested in algebras of the form $\cfalg{\sdp{G}{K}}$.
If we assume that $G$ is abelian, we may utilise \cref{thm:hyper-herz} to find that $\kfalg{G}{K}$ is a natural quotient of $\cfalg{\sdp{G}{K}}$.
\begin{corollary}
    \label{thm:cfalg-quot-kfalg}
    Let $G$ be an abelian compact group, and $K$ a compact group acting on $G$.
    Then
    \begin{equation*}
        \restr{\cfalg{\sdp{G}{K}}}{G} = \kfalg{G}{K}
    \end{equation*}
\end{corollary}
\begin{proof}
    By \cref{thm:hyper-herz}, we have that $\restr{\cfalg{\sdp{G}{K}}}{G} = \kfalg{G}{\sdp{G}{K}}$.
    However, since $G$ is abelian, observe that
    \begin{equation*}
        (gk)h(gk)^{-1}
        = g (\lac{h}{k}) \inv{g}
        =\lac{h}{k}
    \end{equation*}
    for $g,h \in G$ and $k \in K$.
    If we let $\pi_{\sdp{G}{K}}$ and $\pi_K$ be the associated spherical projectors, we then have that $\pi_{\sdp{G}{K}}(gk) = \pi_K(k)$, and so we obtain that $\kfalg{G}{\sdp{G}{K}} = \kfalg{G}{K}$.
\end{proof}

\newcommand{\dpi}{\hat{\pi}}
Let us return to the general locally compact setting, but now suppose that $G$ is an \emph{abelian} group (in the examples we consider, $G$ will either be compact or discrete).
In such a scenario, define the \emph{dual action} of $K$ on $\pdual{G}$ by
\begin{equation}
    \dup{\lac{\chi}{k}}{x} \deq \dup{\chi}{\lac{x}{\inv{k}}}
\end{equation}
for $x \in G$, $\chi \in \pdual{G}$, and $k \in K$.
The actions $K \actson G$ and $K \actson \pdual{G}$ then have associated orbit hypergroups which shall be denoted by $H$ and $\hdual{H}$ respectively.
The symbols $\pi$ and $\dpi$ will denote the corresponding spherical projectors.

As suggested by this notation, it will be shown that the Pontryagin dual of $H$ is indeed $\hdual{H}$.
Hartmann and Lasser \cite{hartmann-lasser} prove that the dual of $H$ can be identified with $\extrem{G,K}$, the extreme points of $\pdefu{G,K}$, where $\pdefu{G,K}$ consists of the unit norm positive-definite $K$-invariant functions on $G$.
Moreover, they show that the hyperproduct is given by pointwise multiplication.
In particular, any probability measure on $\extrem{G,K}$ corresponds to some $\pdefu{G,K}$ function.
This justifies the notation above.
\begin{theorem}
    With $H$ and $\hdual{H}$ as above, the dual of $H$ is in fact $\hdual{H}$.
\end{theorem}
\begin{proof}
    Recall that $\meas{\hdual{H}}$ can be identified with $\kmeas{\pdual{G}}{K}$, and consider the Fourier-Stieltjes transform $\mathcal{F} : \meas{\pdual{G}} \to \cbdd{G}$.
    Bochner's Theorem and \cref{thm:mu-radial} below show that $\mathcal{F}(\prob{\hdual{H}}) = \pdefu{G,K}$.
    This map preserves the extreme points and the hyperproduct (as convolutions become multiplication) so it follows that $\hdual{H}$ is isomorphic to $\extrem{G,K}$.
\end{proof}
\begin{lemma}
    \label{thm:mu-radial}
    Given $K \actson G$, we have that $\mu \in \meas{G}$ is $K$-radial if and only if its Fourier-Stieltjes transform $\ft{\mu} \in \cbdd{\pdual{G}}$ is $K$-radial.
\end{lemma}
\begin{proof}
    We let $\pi$ and $\dpi$ be the spherical projectors of $K \actson G$ and $K \actson \pdual{G}$ respectively.
    For $\mu \in \meas{G}$, observe that
    \begin{align*}
        \dpi(\ft{\mu})(\chi)
        &= \int_K \ft{\mu}(\lac{\chi}{k}) \dee k \\
        &= \int_K \int_G \conj{\dup{\lac{\chi}{k}}{x}} \dee \mu(x) \dee k \\
        &= \int_K \int_G \conj{\dup{\chi}{\lac{x}{\inv{k}}}} \dee \mu(x) \dee k
    \end{align*}
    and if we let $f_\chi(x) = \conj{\dup{\chi}{x}}$, then
    \begin{align*}
        \dpi(\ft{\mu})(\chi)
        &= \int_G \int_K f_\chi(\lac{x}{\inv{k}}) \dee k \dee \mu(x) \\
        &= \int_G \pi(f_\chi)(x) \dee \mu(x) \\
        &= \int_G f_\chi(x) \dee \pi(\mu)(x) \\
        &= \int_G \conj{\dup{\chi}{x}} \dee \pi(\mu)(x)
        = \ft{\pi(\mu)}(\chi)
    \end{align*}
    and from this, it is clear that $\mu$ is $K$-radial if and only if $\ft{\mu}$ is $K$-radial.
\end{proof}

When the modular function on $G$ is $K$-radial, we say that $H$ is \define{ultraspherical}.
Muruganandam \cite[Theorem 3.13]{muruganandam2} shows that all ultraspherical hypergroups are regular Fourier hypergroups, and that moreover, $\kfalg{G}{K} = \falg{H}$.
In particular, this statement always holds for commutative orbit hypergroups.
Furthermore, by \cite[Proposition 2.14]{muruganandam2}, one can show that $\kLo{G}{K} = \Lo{H}$, from which we obtain
\begin{equation}
    \Lo{H} = \kLo{G}{K} = \kfalg{\pdual{G}}{K} = \falg{\hdual{H}}
\end{equation}
where the middle equality holds via the Fourier transform, following a similar argument as given in \cref{thm:mu-radial}.
This holds more generally for strong commutative hypergroups, in that we always have $\Lo{H} = \falg{\hdual{H}}$, see \cite[Proposition 4.2]{muruganandam1}.
With this duality established, we shall attempt to understand the amenability $\falg{H}$ by working with $\cLo{\hdual{H}}$ directly.

\subsection{Amenability criteria}
\sepfootnotecontent{alaghmandan-sketch}{
}

Skantharajah \cite{skantharajah} introduced the notions of amenability via the Reiter conditions $(P_p)$ for hypergroups.
These conditions closely mimic their group analogues, and may be found in Skantharajah's original paper.

A class of Leptin conditions for hypergroups was introduced by Alaghmandan \cite{alaghmandan.leptin}.
In particular, the \emph{1-Leptin condition} $(L_1)$ is defined as follows (recall the definition for the hyperproduct of sets in \eqref{eq:set-hypproduct}).
\begin{definition}
    A hypergroup $H$ satisfies the \define{1-Leptin $(L_1)$} condition if for every compact $C \subseteq H$ and $\varepsilon > 0$, there exists a measurable $V \subseteq H$ such that $0 < \haar_H(V) < \infty$ with $\haar_H(C*V)/\haar_H(V) < 1 + \varepsilon$.
\end{definition}
It was originally shown by Singh \cite[Proposition 4.4.3]{singh} and more generally by Alaghmandan \cite[Proposition 4.1 and Theorem 4.4]{alaghmandan} that $(L_1)$ implies $(P_2)$ for discrete Fourier hypergroups\footnote{
For the curious, Alaghmandan also presents a nice diagram below Example 4.5 summarising some known implications of amenability conditions of discrete Fourier hypergroups.}.
When a hypergroup is $(P_2)$, we have a necessary condition for the amenability of $\Lo{H}$.
This result is stated below, and is presented as Theorem 5.1 in that same paper.
\begin{theorem}[Alaghmandan]
    \label{thm:alaghmandan}
    Let $H$ be a discrete commutative hypergroup which satisfies $(L_1)$.
    If $\lo{H}$ is amenable, then there is some $M \geq 1$ such that $\set{x \in H : \haar_H(x) \leq M}$ is infinite.
\end{theorem}

\section{Examining the affine group \texorpdfstring{$G = \sdp{\Ri^d}{\GL{\Ri}{d}}$}{G = R \textasciicircum d ⋊ GLd(R)}}
\label{sec:affine}

\sepfootnotecontent{poweroverload}{
    As a word of caution, we will be overloading the superscript notation.
    It may mean the algebraic products of ideals, such as in $\Mi^n$, but it may also mean the Cartesian product, such as in $\Ri^d$.
    However, there should be sufficient context to determine which interpretation is in use, so no ambiguity should arise.
}

\subsection{The affine group \texorpdfstring{$G$}{G}}
We now present an application of these results to the central Fourier algebra.
We shall borrow terminology from the theory of local fields, wherein we let $\Fi$ be a non-Archimedean local field, $\Ri$ its ring of integers (a compact discrete valuation ring), and $\Mi$ the unique maximal ideal of $\Ri$.
For readers unfamiliar with this terminology, an example from the $p$-adics suffices without any real loss of generality.
One may assume that $\Fi = \QP$ the $p$-adic numbers, $\Ri = \OP$ the $p$-adic integers, and $\Mi^n = p^n \OP$.\sepfootnote{poweroverload}
Another example is the formal power series over a finite field $\F$, where we set $\Ri = \F[[X]]$, $\Fi = \F((X))$ (the Laurent series over $\F$), and then $\Mi$ consists of polynomials whose constant term is zero with $\Mi^n = X^n \F[[X]]$.

Consider now the action of $K_d = \GL{\Ri}{d}$ on $G_d = \Ri^d$
When $d=1$, one quickly observes that the orbit structure of the action $\Ri^* \actson \Ri$ has a ``layering'': it keeps invariant the nested sets $\Mi^n$, so that the orbits are the shells $\Mi^n \setminus \Mi^{n+1}$.
There is a similar structure in the dual action, so let us give this a name.
\begin{definition}\leavevmode
    \begin{enumerate}[label=(\alph*)]
        \item Let $K$ be a compact group acting on a compact group $G$.
            We say that $K \actson G$ is \define{(compactly) layered by $G_n$} if $G_n$ is a decreasing chain of subgroups whose intersection is trivial, such that the orbits of $K \actson G$ are precisely the sets $A_n \deq G_{n} \setminus G_{n+1}$ and $\set{e}$.
        \item 
            Let $K$ be a compact group acting on a discrete group $\Gamma$.
            We say that $K \actson \Gamma$ is \define{(discretely) layered by $\Gamma_n$} if $\Gamma_n$ is an increasing chain of subgroups whose union is all of $\Gamma$, such that the orbits of $K \actson \Gamma$ are precisely the sets $B_n \deq \Gamma_{n} \setminus \Gamma_{n-1}$.
            For convenience, we assume that $\Gamma_0 = \set{e}$ and $\Gamma_{-1} = \emptyset$.
    \end{enumerate}
\end{definition}
As mentioned, it is straightforward to observe that $K_d \actson G_d$ is compactly layered, and its dual is discretely layered, when $d = 1$.
Let us show that this still holds for arbitrary $d$.
In this case, the layering subgroups will be of the form $\Mi_d^n \deq \set{(x_1, \ldots, x_d) \in \Ri^d : x_i \in \Mi^n}$.
We begin with the following lemma.

\begin{lemma}
    \label{thm:orbits-fd}
    The orbits of the action of $\GL{\Ri}{d}$ on the vector space $\Fi^d$ are precisely the sets $\set{0}$ and $\Mi_d^n \setminus \Mi_d^{n+1}$ for $n \in \Z$.
\end{lemma}
\begin{proof}
    It is straightforward to see that the sets $\Mi_d^{n}$ are invariant under the action of $K$.
    So it remains to show is that their differences are in fact full orbits.
    Without loss of generality, it suffices to prove that $\Mi_d^0 \setminus \Mi_d^1 = \Ri^d\setminus \Mi_d^1$ is an orbit.
    Set $e_1 \in \Ri^d$ to be the usual basis element, and take $x = (x_1, \ldots, x_d) \in \Ri^d \setminus \Mi_d^1$.
    Since $K$ contains the permutation matrices, we may assume without loss of generality that $x_1 \in \Ri \setminus \Mi$.
    Now let $A \in K$ be the matrix
    \begin{equation*}
        A = \begin{bmatrix}
            x_1 & 0 & \cdots & 0 \\
            x_2 & 1 & \cdots & 0 \\
            \vdots & \vdots & \ddots & \vdots \\
            x_d & 0 & \cdots & 1 
        \end{bmatrix}
    \end{equation*}
    where $A$ is invertible since $x_1 \in \Riu$.
    Clearly $A e_1 = x$, and so $x$ is in the orbit of $e_1$.
    By transitivity, we have that $\Ri^d\setminus \Mi_d^1$ is an orbit.
\end{proof}
\begin{proposition}
    \label{thm:dvr-dual-layering}
    $K_d \actson G_d$ is compactly layered and $K_d \actson \pdual{G_d}$ is discretely layered.
\end{proposition}
\begin{proof}
    In both cases, the layering is given as in \cref{thm:orbits-fd}.
    It suffices to observe that $\Ri^d$ is a subgroup of $\Fi^d$, and so the compact layering is given directly by $\Mi_d^n$.
    For the dual version, note that the dual of $\Ri^d$ is isomorphic to $\Fi^d/\Ri^d$, so the discrete layering is given by $\Mi_d^{-n}/\Ri^d$.
\end{proof}

It is clear that hypergroups from compactly layered actions will satisfy $(L_1)$, given that they are already compact.
The same is true for discretely layered actions, though this needs some justification.
The following proposition provides a slight generalisation of this fact. 
\begin{proposition}
    \label{thm:l1-dual-profinite}
    Let $H$ be the orbit hypergroup of $K \actson G$ for a locally compact group $G$.
    If $G$ is the direct limit of $K$-invariant compact open subgroups $G_n$, then $H$ satisfies $(L_1)$.
\end{proposition}
\begin{proof}
    Let $C \subseteq H$ be compact, which we will consider as a $K$-invariant compact subset of $G$.
    Since the subgroups $G_n$ are open, there must be a sufficiently large $N$ such that $C \subseteq G_N$.
    We then have that $C * G_N \subseteq G_N$, and so it follows that $\haar_H(C * G_N)/\haar_H(G_N) \leq 1$.
\end{proof}

This immediately leads us to the main result for discretely layered actions.
\begin{theorem}
    \label{thm:layered-non-amenability}
    Let $\Gamma$ be a discrete abelian group, and suppose that $K \actson \Gamma$ is a discretely layered action.
    If we let $G = \pdual{\Gamma}$, then $\kfalg{G}{K}$ is not amenable.
    In particular, $\cfalg{\sdp{G}{K}}$ is also not amenable.
\end{theorem}
\begin{proof}
    Let $H$ denote the orbit hypergroup for $K \actson \Gamma$.
    By \cref{thm:discrete-hypergroup-dirac}, the elements of $H$ are of the form
    \begin{equation}
        b_n 
        = \frac{1}{\card{B_n}} \indic{B_n}
        = \frac{1}{\card{\Gamma_n} - \card{\Gamma_{n-1}}} \left(\indic{\Gamma_{n}} - \indic{\Gamma_{n-1}}\right)
    \end{equation}
    with Haar measure $\haar_H (b_n) = \card{\Gamma_n} - \card{\Gamma_{n-1}}$.
    These clearly grow without bound, so that for any $M > 0$, the equation $\haar_H(b_n) \geq M$ holds for all but finitely many $n$.
    Since $H$ is $(L_1)$ by \cref{thm:l1-dual-profinite}, then by \cref{thm:alaghmandan} we have that $\Lo{H} = \falg{\hdual{H}} = \kfalg{G}{K}$ is not amenable.
    The final statement follows by \cref{thm:cfalg-quot-kfalg}.
\end{proof}
\begin{corollary}
    $\cfalg{\sdp{\Ri^d}{\GL{\Ri}{d}}}$ is not amenable.
\end{corollary}
This reaffirms the conjecture of Alaghmandan and Spronk.
These are perhaps the simplest examples which are not directly addressed in their original paper, particularly when $d=1$. 

\subsection{On the duality of layerings}
\label{sec:duality-layering}

When $G$ is abelian, it may seem an attractive prospect that if $K \actson G$ is compactly layered by $G_n$, then $K \actson \pdual{G}$ is discretely layered by $\ann{G_n}$.
Certainly, \cref{thm:dvr-dual-layering} shows this to be true for $\GL{\Ri}{d} \actson \Ri^d$.
However, in attempting a proof, surprising complications arise.
We present below a partial proof of this result, though the final step currently remains unsolved (and is a rather interesting problem in its own right).
\begin{conjecture}
    \label{cnj:layer-to-dual}
    Let $G$ be a compact abelian group, and let $K$ be a compact group acting on $G$.
    If $K \actson G$ is compactly layered by $G_n$, then $K \actson \pdual{G}$ is discretely layered by $F_n = \ann{G_n}$.
\end{conjecture}
\begin{proof}[Partial proof]
    Without loss of generality, it suffices to show that $F_1 \setminus \set{e}$ is an orbit.
    To this end, consider the natural action of $K$ on the quotient $G/G_1$.
    By assumption this has two orbits: the trivial orbit and everything else.
    Let us say that such actions are \define{transitive}.

    Notice that $G/G_1$ is finite.
    So by transitivity, every (non-trivial) element must have the same order, and hence it is of the form $\gf{q}^n$ where $\gf{q}$ is the finite field of order $q$.
    We also have that $F_1 \cong \pdual{G/G_1} \cong \pdual{\gf{q}^n} \cong \gf{q}^n$.
    All that remains now is to show that $K \actson F_1$ is also transitive.
\end{proof}
So the question remains, is this dual action also transitive?
In general, it is not too difficult to see that the dual action of $K \actson \gf{q}^n$ is given by $K^T \actson \gf{q}^n$ where $K^T \deq \set{A^T : A \in K}$.
\begin{conjecture}
    Let $G$ be a transitive subgroup of $\GL{\gf{q}}{n}$ acting on $\gf{q}^n$.
    Then the transpose group $G^T$ is also transitive on $\gf{q}^n$.
\end{conjecture}

This question is surprisingly difficult to answer.
There is a classification of all transitive subgroups of $\GL{\gf{q}}{n}$ due to Hering \cite{hering}: in particular there are four infinite classes of transitive subgroups, as well as a handful of sporadic examples.
This is summarised quite succinctly in Appendix 1 of Liebeck's paper \cite{liebeck}.


\section{Acknowledgements}
The author would like to express gratitude to Nico Spronk, for the suggestion of the problem, as well as his guidance and advice. 



\bibliographystyle{plain}
\bibliography{biblio.bib}

\appendix
\section{Hypergroup structure of layered actions}
In preparing this paper, the full hypergroup structure induced by these group had been explicitly computed.
However, it was discovered that this was no longer necessary due to results such as \cref{thm:l1-dual-profinite}.
Nonetheless, its computation yields a relatively simple hypergroup structure, which may be of interest to those studying hypergroups in other contexts.
We provide the details here as a reference.

Suppose that $K \actson \Gamma$ is an action discretely layered by $\Gamma_n$. 
Let $H$ be the associated hypergroup, whose elements $b_n$ correspond to the orbit $B_n = \Gamma_n \setminus \Gamma_{n-1}$.
By \cref{thm:discrete-hypergroup-dirac}, we have that
\begin{equation}
    b_n 
    = \frac{1}{\card{B_n}} \indic{B_n}
    = \frac{1}{\card{\Gamma_n} - \card{\Gamma_{n-1}}} \left(\indic{\Gamma_{n}} - \indic{\Gamma_{n-1}}\right)
\end{equation}
For convenience, we shall let $g_n = \indic{\Gamma_n}$ and $q_n = \card{\Gamma_n}$. 
\begin{proposition}
    \label{thm:hyp-comb}
    Let $H$ be a discrete layered hypergroup coming from the action $K \actson \Gamma$, and assume the notation as given above.
    For $n > m \geq 0$ we have
    \begin{equation*}
        b_n * b_m = b_n
        \qquad \text{and} \qquad
        b_n * b_n = \frac{1}{q_n - q_{n-1}}\left[\sum_{i=0}^{n} (q_{i} - q_{i-1})b_i - q_{n-1}b_n\right]
    \end{equation*}
    where we set $q_{-1} \deq 0$ for convenience.
\end{proposition}
\begin{proof}
    The definition of convolution immediately yields that $g_n * g_m (x) = \card{\Gamma_n \cap x \Gamma_m}$, and so $g_n * g_m = g_m * g_n = g_m$ for $m \leq n$.
    Supposing now that $m < n$, we have
    \begin{align*}
        b_n * b_m
        &= \frac{(g_n - g_{n-1})*(g_m - g_{m-1})}{(q_n - q_{n-1})(q_m - q_{m-1})} \\
        &= \frac{q_m g_n - q_{m-1} g_n - q_m g_{n-1} + q_{m-1}g_{n-1}}{(q_n - q_{n-1})(q_m - q_{m-1})} \\
        &= \frac{(q_m - q_{m-1}) (g_n - g_{n-1})}{(q_n - q_{n-1})(q_m - q_{m-1})}  = b_n
    \end{align*}
    Similarly, we can compute 
    \begin{align*}
        b_n * b_n
        &= \frac{(g_n - g_{n-1})*(g_n - g_{n-1})}{(q_n - q_{n-1})^2} \\
        &= \frac{q_n g_n - 2 q_{n-1} g_n + q_{n-1}g_{n-1}}{(q_n - q_{n-1})^2}\\
        &= \frac{(q_n - q_{n-1}) g_n - q_{n-1} (g_n - g_{n-1})}{(q_n - q_{n-1})^2}\\
        &= \frac{g_n - q_{n-1} b_n}{q_n - q_{n-1}}
    \end{align*}
    Since $b_n$ is a weighted difference of $g_n$ and $g_{n-1}$, we can reverse this to obtain $g_n = \sum_{i=0}^{n} (q_{i} - q_{i-1})b_i$, from which it follows that
    \begin{equation*}
        b_n * b_n = \frac{1}{q_n - q_{n-1}}\left[\sum_{i=0}^{n} (q_{i} - q_{i-1})b_i - q_{n-1}b_n\right]\qedhere
    \end{equation*}
\end{proof}

\end{document}